\documentclass[
final
, nomarks
]{dmtcs-episciences}

\usepackage[utf8]{inputenc}
\usepackage{subfigure}

\usepackage[round]{natbib}

\usepackage{amsmath,amssymb,epsfig,setspace,color,amsthm,enumerate,url,mathrsfs,xcolor, soul}
\newcommand\Tstrut{\rule{0pt}{2.6ex}}         
\newcommand\Bstrut{\rule[-0.9ex]{0pt}{0pt}}   
\usepackage{pdfrender}

\newtheorem{theorem}{Theorem}
\newtheorem{lemma}[theorem]{Lemma}
\newtheorem{observation}[theorem]{Observation}
\newtheorem{question}[theorem]{Question}
\newtheorem{corollary}[theorem]{Corollary}

\author[M.E. Messinger, Amanda Porter]{M.E. Messinger\thanks{supported by NSERC (grant application 2018-04059).}
  \and Amanda Porter\thanks{supported by NSERC (USRA) and Mount Allison University.}}
\title[Eulerian k-dominating reconfiguration graphs]{Eulerian k-dominating reconfiguration graphs}
\affiliation{
 Mount Allison University,\\
 Sackville NB, Canada}
\begin{document}
\publicationdata
{vol. 27:2}
{2025}
{1}
{10.46298/dmtcs.13438}
{2024-04-18; 2024-04-18; 2024-12-13}
{2024-12-16}
\maketitle
\begin{abstract}
For a graph $G$, the vertices of the $k$-dominating graph, denoted $\mathcal{D}_k(G)$, correspond to the dominating sets of $G$ with cardinality at most $k$. Two vertices of $\mathcal{D}_k(G)$ are adjacent if and only if the corresponding dominating sets in $G$ can be obtained from one other by adding or removing a single vertex of $G$.  Since $\mathcal{D}_k(G)$ is not necessarily connected when $k < |V(G)|$, much research has focused on conditions under which $\mathcal{D}_k(G)$ is connected and recent work has explored the existence of Hamilton paths in the $k$-dominating graph.  We consider the complementary problem of determining the conditions under which the $k$-dominating graph is Eulerian.  In the case where $k = |V(G)|$, we characterize those graphs $G$ for which $\mathcal{D}_k(G)$ is Eulerian.  In the case where $k$ is restricted, we determine for a number of graph classes, the conditions under which the $k$-dominating graph is Eulerian. 
\end{abstract}

\section{Introduction}

Reconfiguration examines the relationships between feasible solutions to a problem.  Each feasible solution is represented by a vertex in a reconfiguration graph, where two vertices are adjacent if and only if the associated solutions are related according to some predetermined  rule.  Research has focused on determining structural properties of reconfiguration graphs, with considerable interest in determining the conditions required for a reconfiguration graph to be connected (see the surveys~\cite{survey,Nishimura}).  While many different problems could be modeled with a reconfiguration graph (see~\cite{Nishimura} for examples), to date, most research has considered the underlying problem to be graph-theoretic, such as colouring or domination.  A \emph{dominating} set in a graph $G$ is a set $D \subseteq V(G)$ such that every vertex of $V(G)\backslash D$ is adjacent to a vertex of $D$.  The cardinality of a smallest dominating set of $G$ is called the \emph{domination number} of $G$ and is denoted by $\gamma(G)$.

In this paper,  we consider the \emph{$k$-dominating graph}, which was first introduced in 2014~\cite{HaasSeyffarth}.  Given a graph $G$, each dominating set of cardinality at most $k$, is represented by a vertex in the $k$-dominating graph. Two vertices in the $k$-dominating graph are adjacent if and only if one of the associated dominating sets can be obtained from the other by adding or removing a single vertex. Given a graph $G$, its $k$-dominating graph is denoted by $\mathcal{D}_k(G)$.  In the case where $k=|V(G)|$, we omit the subscript and simply refer to $\mathcal{D}(G)$ as the \emph{dominating graph} of $G$.  The dominating graph of $P_4$ is illustrated in Figure~\ref{fig:D3Pn} (b) and the $k$-dominating graph with $k=3$ is illustrated in Figure~\ref{fig:D3Pn} (c).  

\begin{figure}[htbp]
\[ \includegraphics[width=\textwidth]{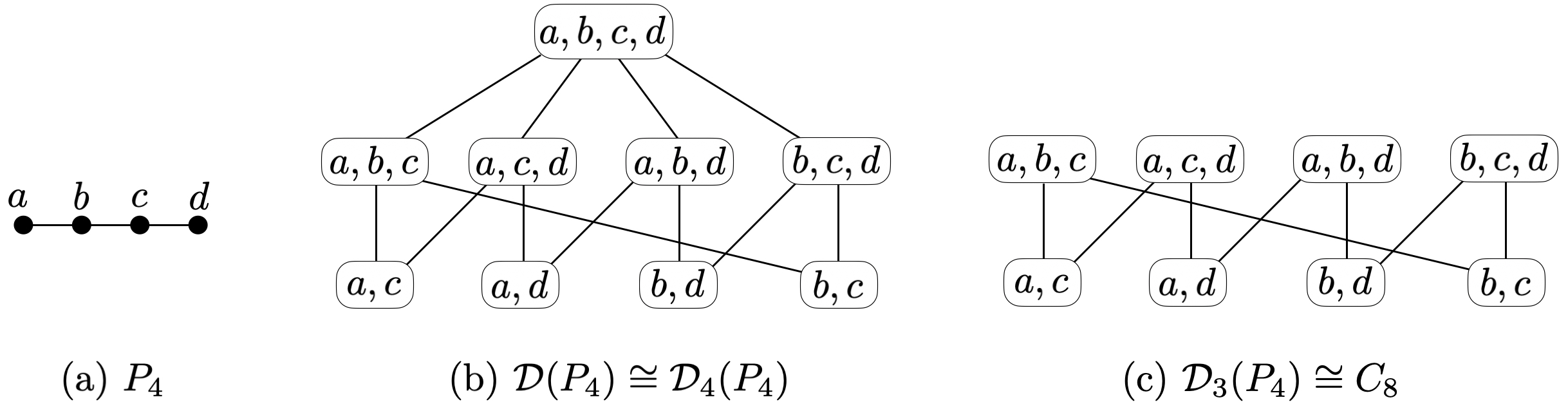} \]

\caption{An example of $k$-dominating graphs of $P_4$.}

\label{fig:D3Pn}
\end{figure}

For a graph $G$, the $k$-dominating graph is always connected when $k= |V(G)|$: for any dominating set, add any unused vertex to obtain another dominating set.   This is not always the case when $k < |V(G)|$; a simple example is $\mathcal{D}_3(K_{1,3})$.  The dominating set consisting of the three vertices of degree one is a minimal dominating set of $K_{1,3}$.  If $k = 3$, then the corresponding vertex in $\mathcal{D}_3(K_{1,3})$ is isolated. Consequently, research has focussed on the conditions required for a $k$-dominating graph to be connected (see~\cite{Mynhardt}, for example).  In terms of complexity, determining whether $\mathcal{D}_k(G)$ is connected, is PSPACE-complete even for graphs of bounded bandwidth, split graphs, planar graphs, and bipartite graphs~\cite{complexity} (see the survey~\cite{Nishimura} for additional algorithmic results).

In~\cite{HaasSeyffarth}, Haas and Seyffarth asked about the conditions required for a $k$-dominating graph to be Hamiltonian.  Adaricheva et al.~\cite{WIGA} showed that no dominating graph contains a Hamilton cycle because dominating graphs are inherently bipartite and contain an odd number of vertices.  Adaricheva et al.~\cite{WIGA2} subsequently proved that the dominating graph of a tree always contains a Hamilton path and that the dominating graph of a cycle on $n$ vertices contains a Hamilton path if and only if $n \not\equiv 0$ (mod $4$).  

Determining whether a graph is Eulerian is generally very easy when compared to determining the existence of a Hamilton path or cycle and this motivates us to ask the following questions.

\begin{question}\label{ques1} For which graphs $G$, is $\mathcal{D}(G)$ Eulerian?\end{question} 

\begin{question}\label{ques2} For which graphs $G$ and integers $k$, where $\gamma(G) < k < |V(G)|$, is $\mathcal{D}_k(G)$ Eulerian?\end{question}

In Section~\ref{sec:one}, we answer Question~\ref{ques1} by providing a characterization of Eulerian dominating graphs.  In Section~\ref{sec:k}, we answer Question~\ref{ques2} for various classes of graphs $G$, including paths, cycles, cocktail party graphs, complete graphs, and complete bipartite graphs.  We conclude with observations relating to well-dominated graphs. \medskip

We end this section with a few definitions and comments on notation.  A graph is \emph{trivial} if it contains no edges. A graph is \emph{Eulerian} if and only if every vertex has even degree and it contains at most one non-trivial component.  To reduce confusion, we refer to graph $G$ as the \emph{seed} graph of $\mathcal{D}_k(G)$.  A dominating set in seed graph $G$ will be denoted by a capital letter and the vertex in $\mathcal{D}_k(G)$ or $\mathcal{D}(G)$ corresponding to that dominating set is denoted by the lowercase letter. Let $S \subseteq V(G)$. If $S$ is a dominating set in $G$ of cardinality at most $k$, then $s$ refers to the corresponding vertex in $\mathcal{D}_k(G)$ and $s \in V(\mathcal{D}_k(G))$.  Let $x,y \in V(G)$. If vertex $x$ is adjacent to vertex $y$, then we say $x$ \emph{dominates} $y$ (and vice versa).  Similarly, if every vertex in $T \subseteq V(G)$ has a neighbour in set $S \subseteq V(G)$, then we say that set $S$ \emph{dominates} $T$.  Finally, we use the notation $[n]$ to denote the set $\{1,\dots,n\}$. 
Any undefined terms can be found in~\cite{wilson}.

\section{Eulerian dominating graphs}\label{sec:one}

In this section, we consider dominating graphs: there is no restriction on the cardinalities of dominating sets in seed graphs.  We begin by stating a useful (unpublished) result due to Brouwer, Csorba, and Schrijver \cite{BCS}.  

\begin{theorem}\label{thm:BCS}\cite{BCS} The number of dominating sets of a finite graph is odd. \end{theorem}

A proof of Theorem~\ref{thm:BCS} can be found in~\cite{WIGA} and it is important to note that the result holds for any finite graph, not simply connected graphs.  Theorem~\ref{thm:BCS} and the Handshaking Lemma\footnote{The Handshaking Lemma~\cite{wilson} states that a graph cannot have an odd number of vertices of odd degree.} imply that every dominating graph has at least one vertex of even degree. 

For a disconnected graph $G$, every dominating set of $G$ is the union of dominating sets of its components. Consequently, the Cartesian product of the dominating graphs of the components precisely describes the structure of $\mathcal{D}(G)$, as stated in Observation~\ref{obs:11}.  The Cartesian product of graphs $H$ and $H'$, denoted $H \hspace{0.5mm} \square \hspace{0.5mm} H'$ is the graph with vertex set $V(H)\hspace{0.5mm}\square \hspace{0.5mm} V(H')$.  Vertices $(u,v)$ and $(x,y)$ in $H \hspace{0.5mm} \square\hspace{0.5mm}  H'$ are adjacent if and only if $u=x$ and $v$ is adjacent to $y$ in $H'$; or $v=y$ and $u$ is adjacent to $x$ in $H$.

\begin{observation}\label{obs:11} Let $G$ be a graph with exactly $m \geq 2$ components $G_1,\dots,G_m$.  Then \begin{displaymath}\mathcal{D}(G_1) ~\square ~\mathcal{D}(G_2) ~\square ~ \cdots ~ \square ~ \mathcal{D}(G_m) \cong \mathcal{D}(G).\end{displaymath} \end{observation}

Given a seed graph $G$ with $m \geq 2$ components $G_1,\dots,G_m$, Observation~\ref{obs:11} implies that $\mathcal{D}(G)$ is Eulerian if and only $\mathcal{D}(G_i)$ is Eulerian for each $i\in [m]$.  To see this, for each $i \in [m]$, let $X_i$ be an arbitrary dominating set of the subgraph of $G$ induced by the vertices of $G_i$.  If $\mathcal{D}(G_i)$ is Eulerian for all $i \in [m]$, then \begin{displaymath}\deg_{\mathcal{D}(G)}\big( (x_1,\dots,x_m)\big) = \deg_{\mathcal{D}(G_1)\square \cdots \square \mathcal{D}(G_m)}\big((x_1,\dots,x_m)\big) = \sum_{i=1}^m \deg_{\mathcal{D}(G_i)}(x_i)\end{displaymath} is even. It was shown in~\cite{HaasSeyffarth} that for any seed graph $G$, $\mathcal{D}(G)$ is connected. Thus $\mathcal{D}(G)$ is Eulerian.  Conversely, suppose $\mathcal{D}(G)$ is Eulerian and, without loss of generality, $\mathcal{D}(G_1)$ is not Eulerian.  Then there exists a vertex $x_1$ with odd degree in $\mathcal{D}(G_1)$.   As noted above, the dominating graph of a connected seed graph contains at least one vertex of even degree.  So for every $i \in \{2,\dots,m\}$ there exists a vertex $y_i$ with even degree in $\mathcal{D}(G_i)$.  By Observation~\ref{obs:11}, if $x_1$ has odd degree in $\mathcal{D}(G_1)$ then the degree of vertex $(x_1,y_2,y_3,\dots,y_m)$ in $\mathcal{D}(G)$ is odd, which yields a contradiction.  Thus, Corollary~\ref{cor:prodEul} follows. 

\begin{corollary}\label{cor:prodEul} Let $G$ be a graph with $m \geq 2$ components $G_1,\dots,G_m$.  Then $\mathcal{D}(G)$ is Eulerian if and only if $\mathcal{D}(G_i)$ is Eulerian for all $i \in [n]$.
\end{corollary}

A consequence of Corollary~\ref{cor:prodEul} is that to complete the characterization of seed graphs $G$ for which $\mathcal{D}(G)$ is Eulerian, we need only consider \emph{connected} seed graphs.  A graph with only one vertex has exactly one dominating set and is therefore trivially Eulerian; thus, the remaining results of this section consider connected seed graphs with at least two vertices.  

Let $G$ be a connected graph on $n \geq 2$ vertices.  Since every set of cardinality $n-1$ dominates $G$, there must exist a minimum non-negative integer $\ell+1$ such that every set of cardinality $\ell+1$ dominates $G$.  Lemma~\ref{lem:part} proves that if some, but not all, sets of cardinality $\ell$ dominate $G$, then $G$ is not not Eulerian.  Theorem~\ref{thm:Eulchar} completes the characterization by considering graphs in which no set of cardinality $\ell$ dominates $G$.

\begin{lemma}\label{lem:part} Let $G$ be a connected graph on $n \geq 2$ vertices.  If for some $\ell \in [n-1]$,

\begin{enumerate}
\item there is a set of cardinality $\ell$ that dominates $G$,

\item there is a set of cardinality $\ell$ that does not dominate $G$, and

\item every set of cardinality $\ell+1$ dominates $G$;
\end{enumerate} then $\mathcal{D}(G)$ has at least one vertex of even degree and at least one vertex of odd degree.\end{lemma} 

\begin{proof} For a contradiction, suppose 1., 2., and 3.\ hold, but the vertex degrees in $\mathcal{D}(G)$ are all of the same parity.  Since $G$ is connected and $\ell < n$, the vertex in $\mathcal{D}(G)$ corresponding to the dominating set that contains all vertices of $G$, will have degree $n$; and the claim follows:\medskip

\noindent \underline{Claim:} \emph{Every vertex degree in $\mathcal{D}(G)$ has the same parity as $n$.}\medskip

By 2., let $A \subset V(G)$ be a set of cardinality $\ell$ that does not dominate $G$.  Consequently, there exists $x \in V(G)$ that is not adjacent to any vertex in $A$.  By 3., $A \cup \{x\}$ dominates $G$, so let $B = V(G)\backslash (A \cup \{x\})$.  Observe that $B \neq \emptyset$; otherwise $x$ is an isolated vertex in $G$.  By 3., $A \cup \{v\}$ dominates $G$ for each $v \in B$.  Since $A$ does not dominate $x$, this implies $x$ must be adjacent to every vertex in $B$.  

Let $I(A)$ denote the set of vertices of $A$ that have no neighbour in $A$.  If a vertex (other than $x$) is removed from $A \cup \{x\}$, then whether the remaining vertices dominate $G$ depend entirely on whether the removed vertex has a neighbour in $A$ or not; i.e. whether or not the removed vertex is in $I(A)$.   We next show that $I(A) \neq \emptyset$ and then to complete the proof, consider two cases: $|I(A)|=1$ and $|I(A)| \geq 2$.

Suppose $I(A) = \emptyset$.  This implies that $A \backslash \{a\}$ dominates $A$ for all $a \in A$. Consequently, $(A \cup \{x\}) \backslash \{a\}$ dominates $G$ for all $a \in A$.  Additionally, $(A \cup \{x\})\cup \{v\}$ dominates $G$ for all $v \in B$.  Therefore, the vertex of $\mathcal{D}(G)$ corresponding to dominating set $A \cup \{x\}$ must have degree $n-1$ in $\mathcal{D}(G)$, which contradicts the claim.  Thus, $I(A)$ is non-empty; let $a \in I(A)$. By 3., $(A \cup \{x,v\})\backslash \{a\}$ dominates $G$ for every $v \in B$.   Since $v \in B$ is arbitrary, this implies every vertex in $I(A)$ must be adjacent to every vertex in $B$; see Figure~\ref{fig:A} where the dotted line represents non-adjacency. 

\begin{figure}[htbp]

\[ \includegraphics[width=0.265\textwidth]{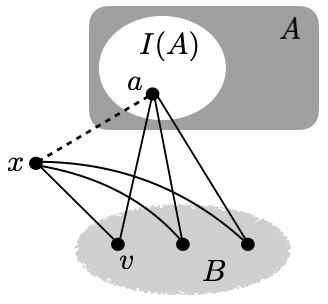} \]

\caption{A visualization of some of the adjacencies (and non-adjacencies) in $G$.}

\label{fig:A}
\end{figure}

\textbf{Case 1.} Suppose $I(A)=\{a\}$.  If $A \backslash \{a\} = \emptyset$ then $A = \{a\}$ and so $\ell=1$.  Then by 1., there exists some vertex $y \in V(G)$ that dominates $G$.  However, the vertex in $\mathcal{D}(G)$, corresponding to dominating set $\{y\}$, has degree $n-1$, which contradicts the claim.  Thus, $|A\backslash \{a\}| \geq 1$.  Let $c$ be a vertex in $B$ that has a neighbour in $A \backslash \{a\}$. (Such a vertex necessarily exists; otherwise $G$ is disconnected.)  By 3., $A \cup \{c\}$ has cardinality $\ell+1$ and dominates $G$.  We list the subsets of $A \cup \{c\}$, of cardinality $\ell$:  \begin{itemize}\setlength\itemsep{-0.1em} \item $A$ -- this set does not dominate $G$ by assumption; \item $(A \cup \{c\})  \backslash \{a'\}$ for each $a' \in A \backslash \{a\}$ -- this set dominates $G$ because $a$ dominates $B$, $c$ dominates $x$, and $A \backslash \{a'\}$ dominates $A$; \item $(A \cup \{c\}) \backslash \{a\}$. \end{itemize}  Observe that $(A \cup \{c\})  \backslash \{a\}$ must not dominate $G$; otherwise the vertex in $\mathcal{D}(G)$ corresponding to dominating set $A \cup \{c\}$ has degree $n-1$, which contradicts the claim.   Since $(A \cup \{c\}) \backslash \{a\}$ does not dominate $G$, Property 1.a.\ must hold:\medskip

\noindent\textbf{Property 1.a.} $\exists c' \in B$ ($c' \neq c$) such that $c'$ has no neighbour in $(A \cup \{c\}) \backslash \{a\}$.\medskip

By considering the subsets of $A \cup \{c'\}$ of cardinality $\ell$, we similarly conclude $(A \cup \{c'\}) \backslash \{a\}$ must not dominate $G$.  However, as $c$ has a neighbour in $A \backslash \{a\}$, this implies Property 1.b.:\medskip

\noindent\textbf{Property 1.b.} $\exists c'' \in B$ ($c'' \neq c'$) such that $c''$ has no neighbour in $(A \cup \{c'\}) \backslash \{a\}$.\medskip

Since $(A \cup \{c,c''\}) \backslash \{a\}$ has cardinality $\ell+1$, it dominates $G$, which tells us that $c'$ has a neighbour in $(A \cup \{c,c''\}) \backslash \{a\}$, which contradicts Property 1.a. and Property 1.b.  Thus, $|I(A)| \neq 1$.\medskip

\textbf{Case 2.} Suppose $|I(A)| \geq 2$.  For a pair $a \in I(A)$ and $v \in B$, we list the subsets of $(A \cup \{x,v\}) \backslash \{a\}$ of cardinality $\ell$: \begin{itemize}\setlength\itemsep{-0.1em} \item $(A \cup \{x\}) \backslash \{a\}$ -- this set does not dominate $G$ because $a \in I(A)$ is not dominated;  \item $(A \cup \{v\}) \backslash\{a\}$ -- this set dominates $G$ because $a' \in I(A)\backslash \{a\}$ dominates $B$, $A \backslash \{a\}$ dominates itself, and $v$ dominates $x$ and $a$; \item  $(A \cup \{x,v\}) \backslash \{a,a'\} $ -- this set dominates $G$ for any $a' \in I(A)\backslash \{a\}$ because $v$ dominates $\{a,a'\}$, $x$ dominates $B$, and $A \backslash \{a,a'\}$ dominates itself; \item  $(A \cup \{x,v\}) \backslash \{a,a^*\}$ dominates $G$ for any $a^* \in A \backslash I(A)$ because $x$ dominates $B$, $v$ dominates $a$, and $A \backslash \{a,a^*\}$ dominates $A \backslash \{a\}$.\end{itemize} Adding any one of $n-(\ell+1)$ vertices to $(A \cup \{x,v\}) \backslash \{a\}$ will result in a dominating set, and removing any one of $\ell$ vertices from $(A \cup \{x,v\}) \backslash \{a\}$ will result in a dominating set. Consequently, the degree of the vertex in $\mathcal{D}(G)$, corresponding to dominating set $(A \cup \{x,v\}) \backslash \{a\}$, is $n-1$, which contradicts the claim.\end{proof} 

For a connected graph $G$ on $n$ vertices, we consider the minimum non-negative integer $\ell+1$ such that every set of cardinality $\ell+1$ dominates $G$.  Then either no set of cardinality $\ell$ dominates $G$ or the conditions of Lemma~\ref{lem:part} are satisfied.  In the latter case, $\mathcal{D}(G)$ is not Eulerian by Lemma~\ref{lem:part}.  Thus, if $\mathcal{D}(G)$ is Eulerian, then there is a minimum positive integer $\ell+1$ such that every set of cardinality $\ell+1$ dominates $G$ and no set of cardinality $\ell$ dominates $G$. We use this implication in Theorem~\ref{thm:Eulchar}, to characterize the connected graphs $G$, for which $\mathcal{D}(G)$ is Eulerian.  

A \emph{perfect matching} in a graph $G$ is a subset $M$ of edges of $G$ such that every vertex in $V(G)$ is adjacent to exactly one edge in $M$.  For an even integer $n \geq 4$, the \emph{cocktail party graph on $n$ vertices}\footnote{The graph represents the handshakes that occur if $n$ couples shake hands with everyone but their partner at a cocktail party.} is the graph obtained by removing a perfect matching from the complete graph $K_n$~\cite{wolf}.  It is also the Tur\'{a}n graph $T(n,n/2)$.  The Tur\'{a}n graph $T(n,r)$ is a complete multipartite graph formed by partitioning $n$ vertices into $r$ subsets, with cardinalities as equal as possible, with two vertices adjacent if and only they belong to different subsets. 

\begin{theorem}\label{thm:Eulchar} Let $G$ be a connected graph on $n \geq 2$ vertices.  Then $\mathcal{D}(G)$ is Eulerian if and only if n is even and $G$ is the cocktail party graph on $n \geq 4$ vertices.\end{theorem} 

\begin{proof} First we consider, all connected graphs on $2$ and $3$ vertices; that is $P_2$, $P_3$, $C_3$ and prove their dominating graphs are not Eulerian. Since a single vertex dominates $P_2$ and the addition of the other vertex will result in a dominating set, the dominating set of cardinality $1$ corresponds to a vertex of degree $1$ in $\mathcal{D}(P_2)$.  The dominating set of cardinality $3$ in $P_3$ and $C_3$ corresponds to a vertex of degree $3$ in the dominating graph because every pair of vertices dominates $P_3$ and $C_3$.  Therefore, no connected graph on $2$ or $3$ vertices has an Eulerian dominating graph.

Let $H_n$ denote the cocktail party graph on $n \geq 4$ vertices where $n$ is even.  Any two vertices dominate $H_n$, but no single vertex dominates $H_n$.  Consequently, each vertex in $\mathcal{D}(H_n)$ corresponds to a subset of $V(H_n)$ of cardinality at least two. Let $T$ be an $r$-element subset of $V(H_n)$ for some $r \in \{2,\dots,n\}$.  If $r > 2$, we can delete a vertex from $T$ and the remaining vertices of $T$ will still dominate $H_n$.  If $r < n$, we can add any of the $n-r$ vertices from $V(H_n) \backslash T$ to $T$ and form a new dominating set of $H_n$.  Thus, $t \in V(\mathcal{D}(H_n))$ will have degree $r + (n-r) = n$ if $2<r<n$.  Furthermore, $t$ has degree $n-2$ if $r=2$ and degree $n$ if $r=n$.  Since $n$ is even, every vertex in $\mathcal{D}(H_n)$ has even degree, and since dominating graphs are always connected, $\mathcal{D}(H_n)$ is Eulerian.\smallskip

For the other direction, let $G$ be a connected graph on $n \geq 4$ vertices such that $\mathcal{D}(G)$ is Eulerian.  Let $X$ denote the dominating set of $G$ that contains all vertices of $G$.  The corresponding vertex $x \in V(\mathcal{D}(G))$ has degree $n$. Since $\mathcal{D}(G)$ is Eulerian, this implies $n$ is even.  

Let $\ell$ be the smallest non-negative integer such that every set of cardinality $\ell+1$ dominates $G$.  Note that $\ell$ necessarily exists because every set of cardinality $n-1$ (and the single set of cardinality $n$) dominates connected graph $G$.  Then either
\begin{enumerate}

\item\label{one} no set of cardinality $\ell$ dominates $G$; or

\item\label{two} there is a set of cardinality $\ell$ that dominates $G$ and there is a set of cardinality $\ell$ that does not dominate $G$.\end{enumerate}

We first note that~\ref{two}.\ cannot occur.  Otherwise, by Lemma~\ref{lem:part}, $\mathcal{D}(G)$ has a vertex of odd degree, which contradicts the assumption that $\mathcal{D}(G)$ is Eulerian.  

Suppose~\ref{one}.\ holds.  Then each vertex in $\mathcal{D}(G)$ that corresponds to a dominating set of cardinality $\ell+1$, has degree $n-(\ell+1)$.  Since $\mathcal{D}(G)$ is Eulerian and $n$ is even, this forces $\ell$ to be odd (and so $\ell \neq 0$).  Suppose $\ell = 1$.  By 1.\, no $1$-element subset of $V(G)$ dominates $G$, but every $2$-element subset of $V(G)$ dominates $G$.  Let $a$ and $b$ be arbitrary non-adjacent vertices of $G$.  Observe that $b$ must be adjacent to every vertex in $V(G) \backslash \{a\}$ since every $2$-element subset of $V(G)$ dominates $G$.  Thus, $b$ has degree $n-1$ in $G$.  By a similar argument, $a$ has degree $n-1$ in $G$.  Since $a$ and $b$ were arbitrary non-adjacent vertices in $G$, this implies $G$ is the complete graph on $n$ vertices with a perfect matching removed.  Thus, $G$ is the cocktail party graph.

Suppose $\ell \geq 3$ and let $S \subset V(G)$ be of cardinality $\ell$.  Since $S$ does not dominate $G$, $\exists~x \in V(G)$ that is not adjacent to any vertex in $S$.  However, as every set of cardinality $\ell+1$ dominates $G$, $x$ must be adjacent to every vertex in $T= V(G)\backslash (S \cup \{x\})$.  Since $\ell+1 \leq n-1$, we know $n-(\ell+1) \geq 1$ and therefore $T \neq \emptyset$.  For any $u \in S$, $(S \cup \{x\}) \backslash \{u\}$ has cardinality $\ell$ and therefore is not a dominating set of $G$.  Since $x$ is adjacent to every vertex in $T$, this implies $u$ has no neighbour in $S \cup \{x\}$.  Since $u$ is an arbitrary vertex in $S$, this implies the subgraph of $G$, induced by $S \cup \{x\}$, is an independent set.  Finally, observe that for any $u \in S$ and $v \in T$,  $(S \cup \{x,v\}) \backslash \{u\}$ has cardinality $\ell+1$ and is therefore a dominating set of $G$.  Since the subgraph induced by $S \cup \{x\}$ is an independent set, every vertex in $S$ must be adjacent to every vertex in $T$; otherwise $u$ is not dominated by $(S \cup \{x,v\})\backslash \{u\}$.  Consequently, for any $v \in T$, the set $\{x,v\}$ dominates $G$. This contradicts the earlier assumption that $\ell \geq 3$. 

Thus, $\ell=1$ and is therefore the cocktail party graph.\end{proof}

\section{Eulerian $k$-dominating graphs}\label{sec:k}

In Section~\ref{sec:one}, we characterized those graphs $G$ for which $\mathcal{D}(G)$ is Eulerian.  In this section, we consider $k<|V(G)|$ and determine, for a variety of classes of graph $G$, the necessary and sufficient conditions on $k$ and $|V(G)|$ for $\mathcal{D}_k(G)$ to be Eulerian.  For a graph $G$, if $k=\gamma(G)$, then $\mathcal{D}_k(G)$ is simply a set of isolated vertices and each component is trivial.  For this reason, and the fact that $\mathcal{D}_k(G) \cong \mathcal{D}(G)$ when $k=|V(G)|$, we assume $\gamma(G) < k < |V(G)|$ throughout this section. We begin by considering two simple graph classes: paths and cycles.  From~\cite{fink}, $\gamma(P_n) = \lceil n/3 \rceil$ for $n \geq 1$ and $\gamma(C_n) = \lceil n/3 \rceil$ for $n \geq 3$.  

\begin{table}[htb]
\[ 
\setlength\arraycolsep{3.5pt}
\begin{array}{|c||c|c|c|}
\hline
\hline
n & k & \mathcal{D}_k(P_n) \text{ Eulerian?} & \mathcal{D}_k(C_n) \text{ Eulerian?}\Tstrut\Bstrut\\
\hline 
\hline
~n \equiv 2 \text{ (mod } 3) \text{ and } n \geq 2~  & ~\lceil \frac{n}{3}\rceil +1~ & \times  & \times \Tstrut\\
~ & \vdots & \vdots & \vdots\\
~ & n-1 & \times & \times \Bstrut\\[3pt]
\hline
n \not\equiv 2 \text{ (mod} 3) \text{ and } n \geq 8 & \lceil \frac{n}{3}\rceil +1 & \times & \times \Tstrut \Bstrut\\[4pt]
\hline
n \not\equiv 2 \text{ (mod} 3) \text{ and } n \geq 8  & \lceil \frac{n}{3}\rceil +2 & \times & \times \Tstrut\\
~ & \vdots & \vdots & \vdots \\
~ & n-1 & \times & \times \Bstrut\\[3pt]
\hline
3 & 2 & \times & \checkmark \Tstrut\Bstrut\\[3pt]
\hline
4 & 3 & \checkmark & \times\Tstrut\Bstrut\\[3pt]
\hline
6 & 3 & \times  & \times \Tstrut\\
~ & 4 & \times  & \times \\
~ & 5 & \times & \times \Bstrut\\[3pt]
\hline
7 & 4 & \times  & \checkmark \Tstrut\\
~ & 5 & \times & \times \\
~ & 6 & ~\times & \times \Bstrut\\[3pt]
\hline
\end{array}
\]
\label{Table:PnCn}

\caption{The different cases considered in the proof of Theorem~\ref{thm:pathcycle}.}
\end{table}

\begin{theorem}\label{thm:pathcycle} Let $n$ and $k$ be integers such that $n \geq 3$ and $\lceil n/3 \rceil < k < n$.  Then \begin{itemize} \item $\mathcal{D}_k(P_n)$ is Eulerian if and only if $n=4$ and $k=3$; and  \item $\mathcal{D}_k(C_n)$ is Eulerian if and only if either $n=7$ and $k=4$, or $n=3$ and $k=2$.\end{itemize}  \end{theorem}

\begin{proof} We consider cases: $n \equiv 2$ (mod $3$); then $n \not\equiv 2$ (mod $3$) and $n \geq 8$; and finally, $n \not\equiv 2$ (mod $3$) and $n \leq 8$.  These cases are summarized in Table~\ref{Table:PnCn}. Vertices in $\mathcal{D}_k(P_n)$ and $\mathcal{D}_k(C_n)$ that correspond to minimum dominating sets in $P_n$ and $C_n$, have degree $(n-\lceil n/3 \rceil)$, which is odd when $n \equiv 2$ (mod $3$).  Thus, neither $\mathcal{D}_k(P_n)$ nor $\mathcal{D}_k(C_n)$ is Eulerian when $n \equiv 2$ (mod $3$).

We now suppose $n \not\equiv 2$ (mod $3$). Let $n = 3\ell+j$ for some integers $\ell,j$ such that $\ell \geq 3$ and $j \in \{0,1\}$.  We will consider cases $\ell=1$ and $\ell=2$ (which correspond to $n \leq 8$) at the end.   Let $P_n = (v_0,v_1,\dots,v_{n-1})$ and $C_n = (v_0,v_1,\dots,v_{n-1})$, with adjacencies modulo $n$ for $C_n$.  Let \begin{displaymath}S = \Big\{v_{3i+1} : 0 \leq i \leq \lceil (n-4)/3 \rceil \Big\} \cup \begin{cases} \emptyset & \text{ if $n \equiv 0$ (mod $3$)} \\ \{v_{n-2}\} & \text{ if $n \equiv 1$ (mod $3$).} \end{cases}\end{displaymath} as highlighted in Figure~\ref{fig:sets} (a) for $n \equiv 0$ (mod $3$) and (b) for $n \equiv 1$ (mod $3$).  Observe that $S$ is a minimum dominating set on both $P_n$ and $C_n$.  Let $S' = S \cup \{v_3\}$.   If $k = \lceil n/3 \rceil +1$, then $s'$ has degree $1$ in both $\mathcal{D}_k(P_n)$ and $\mathcal{D}_k(C_n)$ as $v_3$ is the only vertex of $S'$ whose removal yields a dominating set.  Therefore, neither $\mathcal{D}_k(P_n)$ nor $\mathcal{D}_k(C_n)$ is Eulerian when $k = \lceil n/3 \rceil+1$.  Let $S'' = (S \backslash \{v_4\}) \cup \{v_3,v_5\}$ and note that $S''$ is a minimal dominating set in both $P_n$ and $C_n$.  If $k > \lceil n/3 \rceil+1$, then $s''$ has degree $(n-\lceil n/3 \rceil-1)$ in both $\mathcal{D}_k(P_n)$ and $\mathcal{D}_k(C_n)$.  The quantity $(n-\lceil n/3 \rceil-1)$ is odd when $n \equiv 0,1$ (mod $3$).  Therefore, neither $\mathcal{D}_k(P_n)$ nor $\mathcal{D}_k(C_n)$ is Eulerian for $k > \lceil n/3 \rceil$. 

\begin{figure}[htbp] 
\[ \includegraphics[width=0.8\textwidth]{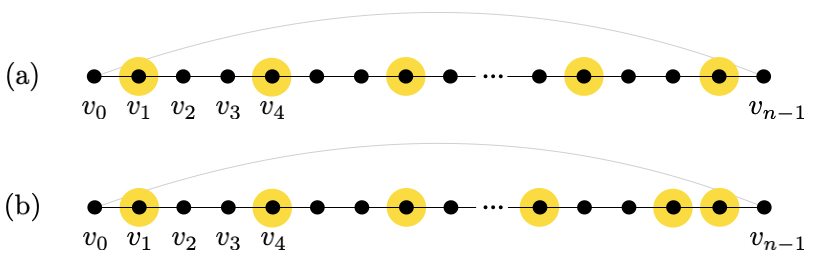} \] 

\caption{Dominating sets on paths and cycles with $n \equiv 0,1$ (mod $3$).}

\label{fig:sets}
\end{figure} 

It remains to consider the cases where $n \equiv 0,1$ (mod $3$) and $\ell \in \{1,2\}$; that is $n \in \{3,4,6,7\}$.  Every vertex in $\mathcal{D}_2(C_3)$ has degree two: removing one of the two vertices in a dominating set of $C_3$ of cardinality two results in a dominating set; adding one of the two vertices not in a dominating set of cardinality one results in a dominating set.  However, $\mathcal{D}_2(P_3)$ is not Eulerian as the dominating set $\{v_0,v_1\}$ corresponds to a vertex of degree $1$ in $\mathcal{D}_2(P_3)$.  From Figure~\ref{fig:D3Pn}, we see $\mathcal{D}_3(P_4)$ is Eulerian.  However, $\mathcal{D}_3(C_4)$ is not Eulerian: the dominating set $\{v_0,v_1,v_2\}$ corresponds to a vertex of degree $3$ in $\mathcal{D}_3(C_4)$ as the removal of any vertex from the set leaves a dominating set.  We note that sets $S,S',S''$ and the argument from the previous paragraph can be applied to both $P_6$ and $C_6$ in order to conclude that $\mathcal{D}_k(P_6)$ and $\mathcal{D}_k(C_6)$ are not Eulerian for $k \in \{3,4,5\}$.

Finally, for $n=7$ and for both $P_7$ and $C_7$, let $S_7$ denote the dominating set in Figure~\ref{fig:P7} (a) (for $P_7$, we ignore the curved edge).  Note that $S_7 \backslash \{v_3\}$ and $S_7 \backslash \{v_4\}$ are dominating sets on both $P_7$ and $C_7$.  If $k \in \{5,6\}$, then adding one of three vertices to $S_7$ forms a dominating set on both $P_7$ and $C_7$.  Thus $s_7$ has degree $5$ in both $\mathcal{D}_k(P_n)$ and $\mathcal{D}_k(C_n)$. It remains to consider $\mathcal{D}_4(P_7)$ and $\mathcal{D}_4(C_7)$.  For $P_7$, we consider the dominating set given Figure~\ref{fig:P7} (b).  The corresponding vertex in $\mathcal{D}_4(P_7)$ has degree $1$ because the only vertex that can be removed from the set and leave a dominating set is $v_0$.  To conclude the proof, we next show that $\mathcal{D}_4(C_7)$ is Eulerian.

Every dominating set of cardinality $3$ in $C_7$ is minimal and therefore corresponds to a vertex of degree $4$ in $\mathcal{D}_4(C_7)$ (as any one of $4$ vertices can be added to the set).  Every dominating set of cardinality $4$ contains at least $2$ adjacent vertices.  Let $\mathcal{S}$ denote the collection of dominating sets of $C_7$ with cardinality $4$ that contain $v_0$ and $v_1$.  Each set in $\mathcal{S}$ contains an additional two of the remaining $5$ vertices.  Up to a relabelling of vertices, there are three sets in $\mathcal{S}$; these are highlighted in Figure~\ref{fig:c7}.  By inspection, we see that for each set in $\mathcal{S}$, one of two vertices can be removed to leave a dominating set, thus, each dominating set in $\mathcal{S}$ corresponds to a vertex of degree two in $\mathcal{D}_4(C_7)$.
\end{proof}

\begin{figure}[htbp] 
\[ \includegraphics[width=0.65\textwidth]{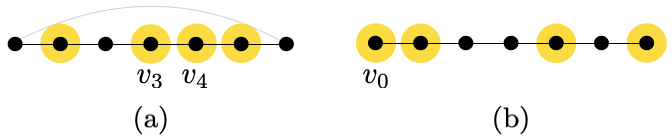} \] 

\caption{Dominating sets on $P_7$.}

\label{fig:P7}
\end{figure}  

\begin{figure}[htbp] 
\[ \includegraphics[width=\textwidth]{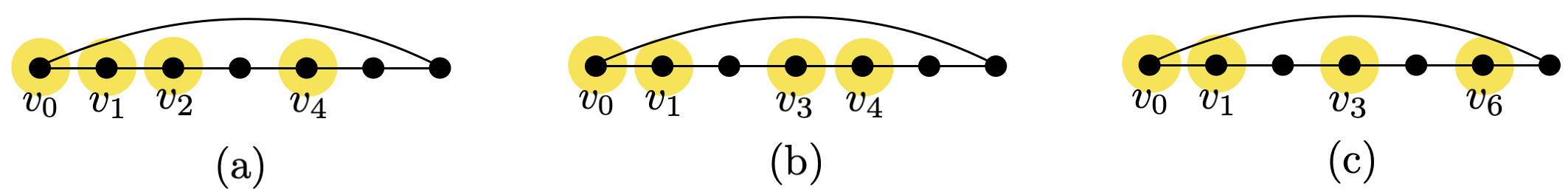} \] 

\caption{Sets of cardinality $4$ on $C_7$ that contain $v_0$ and $v_1$.}

\label{fig:c7}
\end{figure} 

We next characterize the complete bipartite graphs whose $k$-dominating graphs are Eulerian.  

\begin{theorem}\label{thm:compbip} Let $m$, $n$, and $k$ be integers such that $m \leq n$ and $\gamma(K_{m,n}) < k < m+n$.  Then $\mathcal{D}_k(K_{m,n})$ is Eulerian if and only if \begin{itemize} \item $m=1$, $n$ is even, and $k$ is odd; or \item $m \geq 3$, $m \equiv n$ {\rm(mod} $2${\rm)}, and $k=3$.\end{itemize}
\end{theorem} 

\begin{proof} Label the vertices in the partite sets of $K_{m,n}$ as $X = \{x_1,\dots,x_m\}$  and $Y = \{y_1,\dots,y_n\}$.  First, suppose $m=1$.  The only dominating set that has cardinality less than $n+1$ and does not contain $x_1$ is the set $Y$, which is a minimal dominating set.  The corresponding vertex in $\mathcal{D}_n(K_{1,n})$ is an isolated vertex.  We next consider dominating sets of $K_{1,n}$ that have cardinality less than $n+1$ and contain $x_1$.  Let $A$ be such dominating set.  The addition of any vertex in $Y \backslash A$ to $A$; or the removal of any vertex from $Y \cap A$ from $A$ yields a dominating set.  Thus, if $k > |A|$, vertex $a$ has degree $n$ in $\mathcal{D}_k(K_{1,n})$ and if $k = |A|$, vertex $a$ has degree $(|A|-1)=k-1$ in $\mathcal{D}_k(K_{1,n})$. Consequently, $\mathcal{D}_k(K_{1,n})$ is Eulerian precisely when $n$ is even and $k$ is odd.

Second, suppose $m=2$. As $K_{2,2} \cong C_4$, by Theorem~\ref{thm:pathcycle}, $\mathcal{D}_k(K_{2,2})$ is not Eulerian.  Thus, we assume $n>m=2$; consequently, $k \geq 3$.  Let $S = \{x_1,x_2,y_1\}$ and observe that the removal of any vertex from $S$ yields a dominating set.  So if $k=3$, $s$ has degree $3$ in $\mathcal{D}_3(K_{2,n})$, which is therefore not Eulerian.  For $k\geq 4$, let $T = \{x_1,y_1,y_2\}$. Then $t$ has degree $n+1$ (from $T$: add $x_2$, remove one of $y_1,y_2$, or add a vertex from $Y \backslash \{y_1,y_2\}$).  Dominating set $\{x_1,x_2\}$ corresponds to a vertex of degree $n$ in $\mathcal{D}_k(K_{2,n})$, so $\mathcal{D}_k(K_{2,n})$ has vertices of both degree $n+1$ and $n$ when $k \geq 4$. Thus $\mathcal{D}_k(K_{2,n})$ is not Eulerian for $n \geq m = 2$ and $\gamma(K_{2,n}) < k < m+n$.

Third, suppose $n \geq m \geq 3$.  Dominating set $\{x_1,y_1\}$ corresponds to a vertex of degree $m+n-2$ in $\mathcal{D}_k(K_{m,n})$.  As a result, for $\mathcal{D}_k(K_{m,n})$ to be Eulerian, $m$ and $n$ must be of the same parity.  If $k>3$, consider $T = \{x_1,x_2,y_1\}$.  Then $t$ has degree $m+n-1$ in $\mathcal{D}_k(K_{m,n})$: we can add one of $(m-2)+(n-1)$ vertices to $T$ or remove one of $x_1,x_2$ from $T$ and the resulting set is a dominating set.  Thus, $\mathcal{D}_k(K_{m,n})$ is not Eulerian for $n \geq m \geq 3$ and $k>3$.  Assume $k=3$ and let $A$ be a dominating set of $K_{m,n}$ that contains, without loss of generality, two vertices of $X$ and one vertex of $Y$.  Then $a$ has degree $2$ in $\mathcal{D}_3(K_{m,n})$ because one vertex of $A \cap X$ can be removed from $A$ to yield a dominating set, but the removal of the vertex $A \cap Y$ from $A$ does not yield a dominating set.  If $n>m=3$, the dominating set $\{x_1,x_2,x_3\}$ corresponds to an isolated vertex in $\mathcal{D}_3(K_{3,n})$.  Similarly, if $n=m=3$, dominating set $\{y_1,y_2,y_3\}$ corresponds to an isolated vertex in $\mathcal{D}_3(K_{3,3})$.  Thus $\mathcal{D}_3(K_{m,n})$ is Eulerian provided $m$ and $n$ are of the same parity and $n \geq m \geq 3$.\end{proof}
 
Motivated by the fact that the dominating graph of a cocktail party graph is Eulerian, we next consider $k$-dominating graphs of cocktail party graphs.

\begin{theorem}\label{D_kCocktail} Let $H_n$ be the cocktail party graph on $n \geq 4$ vertices where $n$ is even. For any integer $k$ such that $2 < k < n$, $\mathcal{D}_k(H_n)$ is Eulerian if and only if $k$ is even.\end{theorem}

\begin{proof} Let $\mathcal{S}$ be the collection of $p$-element subsets of $V(H_n)$ for $p \in \{2,\dots,k\}$.  Since every pair of vertices dominate $H_n$, but no single vertex dominates $H_n$, the elements of $\mathcal{S}$ are precisely the dominating sets of $H_n$.  Let $S$ be an arbitrarily element of $\mathcal{S}$; then $s$ is the vertex in $\mathcal{D}_k(H_n)$ corresponding to set $S$.  

If $p=2$, then $s$ has degree $n-2$ as any vertex of $V(H_n)\backslash S$ can be added to $S$.  If $p \in \{3,\dots,k-1\}$, then $s$ has degree $n$ because every subset of cardinality $p-1$ dominates $H_n$ and, adding one of the $n-p$ vertices of $V(H_n)\backslash S$ to $S$ will form a dominating set of cardinality $p+1 \leq k$.  Finally, if $p=k$, every subset of $S$ with cardinality $p-1$ dominates $H_n$ and so $s$ has degree $p=k$.  Thus, $\mathcal{D}_k(H_n)$ is Eulerian if and only if $k$ is even. \end{proof}

For the cocktail party graph $H_n$, every set of cardinality $\gamma(H_n)$ dominates $H_n$.  This motivates us to consider complete graphs, $K_n$, as they have the same property: every set of cardinality $\gamma(K_n)$ dominates $K_n$.  For $n \geq 1$ and $k \geq 2$, certainly $\mathcal{D}_k(K_n)$ is not Eulerian when $n$ is even; otherwise the vertex in $\mathcal{D}_k(K_n)$ corresponding to a dominating set of cardinality $1$ has degree $n-1$, which is odd.  If $n$ is odd and $k>2$ the vertex in $\mathcal{D}_k(K_n)$, corresponding to a dominating set of cardinality $2$, has degree $n$, which is odd.  Finally, for odd $n$, $\mathcal{D}_2(K_n)$ contains only vertices of degree $2$ and $n-1$ and hence, is Eulerian.

\begin{observation}\label{obs:1} Let $n$ and $k$ be integers such that $n \geq 2$ and $1 < k < n$.  Then $\mathcal{D}_k(K_n)$ is Eulerian if and only if $n$ is odd and $k =2$.\end{observation}

\begin{corollary} Let $G$ be a connected graph on $n > 1$ vertices with the property that every set of cardinality $\gamma(G)$ dominates $G$.  For integer $k$ such that $\gamma(G) < k < n$, $\mathcal{D}_k(G)$ is Eulerian if and only if \begin{enumerate} \item $G$ is a complete graph on $n > 1$ vertices where $n$ is odd and $k=2$; or \item $G$ is a cocktail party graph on $n\geq 4$ vertices and $k$ is even. \end{enumerate} \end{corollary}

\begin{proof} Let $G$ be a connected graph on $n > 1$ vertices with the property that every set of cardinality $\gamma(G)$ dominates $G$.  If $\gamma(G)=1$, then $G$ must be a complete graph and 1.\ follows from Observation~\ref{obs:1}.  If $\gamma(G)=2$, then $G$ must be the cocktail party graph on $n > 4$ vertices and 2.\ follows from Theorem~\ref{D_kCocktail}.

Suppose $\gamma(G) \geq 3$ and let $A$ be a set of cardinality $\gamma(G)-1$.  Since $|A| < \gamma(G)$, there is at least one vertex of $G$ that is not adjacent to any vertex in $A$.  Let $x$ be such a vertex.  If $A = V(G)\backslash \{x\}$ then $G$ is disconnected, so $B = V(G) \backslash (A \cup \{x\})$ must be nonempty.  Since $A \cup \{v\}$ for any $v \in B$ has cardinality $\gamma(G)$, it must dominate $G$.  Consequently, $x$ must be adjacent to every vertex in $B$.  Let $a \in A$.  We know $A \cup \{x\}$ dominates $G$ due to its cardinality and $(A \cup \{x\}) \backslash \{a\}$ does not dominate $G$ due to its cardinality.  Since $x$ is adjacent to every vertex in $B$, it must be that $A \cup \{x\} \backslash \{a\}$ does not dominate $a$.  Consequently $a$ has no neighbours in $A \backslash \{a\}$ and we conclude the vertices of $A$ form an independent set.

For all $a \in A$, $v \in B$, set $(A \cup \{v,x\}) \backslash \{a\}$ dominates $G$ since it has cardinality $\gamma(G)$. This implies every vertex in $A$ is adjacent to every vertex in $B$.  Observe however, that $A \cup \{v\} \backslash \{a\}$ dominates $G$ for any $a \in A, v \in B$, because $A \backslash \{a\}$ is nonempty and dominates $(A \backslash \{a\}) \cup B$; while $v$ dominates $\{x,a\}$.  This yields a contradiction as $|A \cup \{v\} \backslash \{a\}|=\gamma(G)-1$.\end{proof}

The \emph{upper domination number} of a graph $G$, denoted $\Gamma(G)$, is the maximum cardinality of a minimal dominating set in $G$.  In both complete graphs and cocktail party graphs, the domination number equals the upper domination number.  These are examples of well-dominated graphs.  Well-dominated graphs were introduced in~\cite{FinbowEtAl} where they define a graph $G$ to be \emph{well-dominated} if $\gamma(G) = \Gamma(G)$. The reader may also note that the paths and cycles with Eulerian $k$-dominating graphs, described in Theorem~\ref{thm:pathcycle} are also well-dominated seed graphs.  As pointed out in~\cite{Anderson}, it is easy to show that the corona of any graph is well-dominated. For any graph on $n \geq 2$ vertices, the \emph{corona} of graph $G$, denoted by $G \circ K_1$, is the graph on $2n$ vertices obtained by joining a vertex of degree one to each vertex of $G$.  A vertex of degree one is called a \emph{pendant vertex} and an edge incident with a pendant vertex is called a \emph{pendant edge}.  We first note that if $G$ is a graph on $n \geq 2$ vertices, then $\gamma(G \circ K_1) = \Gamma(G \circ K_1)=n$: every minimal dominating set must contain exactly one endpoint of each pendant edge.  

\begin{theorem}\label{lem:corona} Let $G$ be a graph on $n \geq 2$ vertices and $k$ be an integer for which $n < k < 2n$. Then $\mathcal{D}_k(G \circ K_1)$ is Eulerian if and only if $n$ is even and $k=n+1$.\end{theorem}

\begin{proof} Let $G$ be a graph on $n \geq 2$ vertices.  Let $S$ be a minimum dominating set on $G \circ K_1$.  Then $|S|=n$. For $k > n$, the degree of $s$ in $\mathcal{D}_k(G \circ K_1)$ is $|V(G \circ K_1)|-|S| = n$.  Thus, $s$ has even degree if and only if $n$ is even.

Suppose $n$ is even and let $S' = S \cup \{v\}$ for some $v \in V(G \circ K_1) \backslash S$.  Since $S$ contains exactly one endpoint of each pendant edge in $G \circ K_1$, $S'$ contains both endpoints of exactly one pendant edge $uv$ in $G \circ K_1$.  Removing exactly one of $\{u,v\}$ from $S'$ yields a dominating set of $G \circ K_1$, but for $w \neq u$, $w \neq v$, set $S' \cup \{v\} \backslash \{w\}$ does not dominate $G \circ K_1$.  Thus, the degree of $s'$ in $\mathcal{D}_{n+1}(G \circ K_1)$ is two; and for even $n$, $\mathcal{D}_{n+1}(G \circ K_1)$ is Eulerian.  Note, however, that if $k > n+1$, vertex $s$ in $\mathcal{D}_k(G \circ K_1)$ will have degree $n+1$ since we can remove one of $u,v$ from $S'$ or add one of the other $n-1$ vertices to $S'$ and form a dominating set on $G \circ K_1$.  Thus, if $k>n+1$, $\mathcal{D}_k(G \circ K_1)$ is not Eulerian.\end{proof}

It known from~\cite{FinbowEtAl} that for a connected bipartite graph $G$; $G$ is well-dominated if and only if $G \cong C_4$ or $G$ is the corona of a graph.  Combining this with Theorem~\ref{lem:corona} yields the following result.

\begin{corollary} Let $G$ be a bipartite, well-dominated graph on $2n$ vertices.  Then $\mathcal{D}_k(G)$ is Eulerian if and only if 
\begin{itemize} \item $G \cong C_4$ and $k=3$; or \item $G$ is the corona of a graph, $n \equiv 0$ {\rm (mod} $2${\rm)}, and $k= n+1$.\end{itemize}\end{corollary}

A complete characterization for well-dominated graphs with odd cycles has not yet been found; see~\cite{Anderson} for recent advances on well-dominated graphs containing odd cycles.  The cocktail party graphs (on at least $6$ vertices), complete graphs, and $C_7$ are examples of well-dominated graphs that contain odd cycles and have Eulerian $k$-dominating graphs (for particular values of $k$).  

\begin{question} Which well-dominated graphs contain an odd cycle and have an Eulerian $k$-dominating graph?\end{question}

To conclude, we restate Question~\ref{ques2} since it remains only partially answered.\medskip

\noindent \textbf{Question 2.} For which graphs $G$ and integers $k$, where $\gamma(G) < k < |V(G)|$, is $\mathcal{D}_k(G)$ Eulerian?

\bibliographystyle{abbrvnat}
\bibliography{Reconfig-dmtcs}
\label{sec:biblio}

\end{document}